\documentclass[letterpaper, 10 pt, conference]{ieeeconf}  %\renewcommand{\baselinestretch}{0.98}

\IEEEoverridecommandlockouts                             
\usepackage{float}
\usepackage{verbatim}
\usepackage{paralist}
\usepackage{amsfonts}
\usepackage{graphics} % for pdf, bitmapped graphics files
\usepackage{epsfig} % for postscript graphics files
\usepackage{mathptmx} % assumes new font selection scheme installed
\usepackage{times} % assumes new font selection scheme installed
\usepackage{amsmath} % assumes amsmath package installed
\usepackage{amssymb}  % assumes amsmath package installed
\usepackage{physics}
\usepackage{mdwmath}
\usepackage{mdwtab}
\usepackage{color}
\usepackage[hidelinks]{hyperref}
\usepackage{algorithm}
\usepackage{algorithmic}

\usepackage{caption}
\usepackage{subcaption}
\usepackage{amsmath} % Make sure this is included
\usepackage{bm}      % Package for bold math symbols

\usepackage{tikz}
\usetikzlibrary{shapes, arrows.meta, positioning}

\newtheorem{proposition}{Proposition}% 

\newtheorem{remark}{Remark}

\def\BibTeX{{\rm B\kern-.05em{\sc i\kern-.025em b}\kern-.08em
    T\kern-.1667em\lower.7ex\hbox{E}\kern-.125emX}}

\begin{document}

\title{Tensor Network Methods for Advection–Diffusion–Reaction Systems Using Quantum-Inspired Representations}

\author{Nahid Binandeh Dehaghani, Rafal Wisniewski, A. Pedro Aguiar
\thanks{Nahid Binandeh Dehaghani is with the Department of Electronic Systems, Aalborg University, Aalborg, Denmark.
        {\tt\small nahidbd@es.aau.dk}}%
\thanks{Rafal Wisniewski is with the Department of Electronic Systems, Aalborg University, Aalborg, Denmark. {\tt\small raf@es.aau.dk}}%} 
\thanks{A. Pedro Aguiar is with SYSTEC-ARISE, Faculty of Engineering, University of Porto, Porto, Portugal. {\tt\small 	pedro.aguiar@fe.up.pt}}
%} 
\thanks{The authors acknowledge the support of the Danish e-Infrastructure Consortium (DeiC) and the National Quantum Algorithm Academy (NQAA) through the Postdoctoral Scholarship 
under the project ``Quantum-Driven Solutions for Multi-Agent Systems and Advanced Computation''. This work was also partially supported by the Research Center for Systems and Technologies (SYSTEC, DOI 10.54499/UID/00147/2025) - and the Associate Laboratory Advanced Production and Intelligent Systems (ARISE, DOI 10.54499/LA/P/0112/2020), both funded by Fundação para a Ciência e a Tecnologia, I.P./ MCTES through the national funds.}
}

\maketitle

\begin{abstract}
We present a quantum-inspired tensor-network framework for solving
advection--diffusion--reaction (ADR) partial differential equations.
Discretized solution fields are encoded as matrix product states (MPS),
while differential operators are represented as matrix product operators
(MPOs). Time integration is performed entirely in tensor-network form using
explicit Euler updates with controlled truncation.
The method is evaluated on one- and two-dimensional ADR problems and
compared with high-accuracy Runge--Kutta reference solutions. Numerical
results show that the proposed representation remains compact, stable, and
accurate across a range of dynamical regimes. The solver captures both
local solution profiles and global observables while maintaining small bond
dimensions throughout the simulation. These results highlight the potential
of tensor networks as efficient structure-preserving tools for PDE
simulation in multiple spatial dimensions.
\end{abstract}

\section{Introduction}

Advection--diffusion--reaction (ADR) equations arise in a wide range of
physical, chemical, and biological systems, including pollutant transport,
heat transfer, reactive flows, and ecological population dynamics
\cite{hundsdorfer2003numerical}. Despite their relatively simple mathematical
form, ADR systems frequently exhibit a combination of sharp advective fronts,
diffusion-driven smoothing, and stiff reaction dynamics. These multiscale
behaviors pose considerable difficulties for classical numerical solvers,
particularly when high spatial resolution or higher spatial dimensions are
required.

Tensor-network representations offer a principled mechanism to mitigate
the curse of dimensionality associated with high-resolution discretizations,
especially when the %underlying 
solution admits a low-rank tensor structure.
Originally developed in quantum many-body physics for the efficient
representation of high-dimensional wavefunctions
\cite{TN_Review,Schollwock2011,Orus2014}, tensor networks, and matrix product
states (MPS) in particular, have emerged as powerful tools in numerical
scientific computing. Their ability to exploit low-rank separability and
localized correlation structure enables compact representations of
high-dimensional state vectors and structured operators
\cite{khoromskij2011d}. Matrix product operators (MPOs) extend this
framework to linear operators arising from discretized PDEs.

Recent work has demonstrated that quantum-inspired tensor networks (QTNs)
can approximate the flow maps of hydrodynamic PDEs with strong compression
and controlled error growth \cite{Dehaghani2026}. These results suggest that
many discretized PDE solution manifolds possess effective low-rank structure
after tensorization, allowing the underlying dynamics to be simulated with
computational cost governed primarily by the bond dimensions and tensor-network
structure rather than by the full dense grid size.

Building on these developments, we develop a tensor-network time-stepping framework for ADR equations. Discretized solution fields are encoded as MPS, while finite-difference derivative operators are represented as MPOs. Time integration is performed entirely in tensor-network form using explicit Euler updates with canonicalization and SVD-based truncation, enabling compact and efficient simulation of high-resolution discretizations.

\paragraph*{Contributions}
This work introduces a tensor-network framework for solving
ADR equations using MPS/MPOs. We construct compact MPO representations of
finite-difference advection and diffusion operators and perform explicit
Euler time stepping entirely within the tensor-network format. Numerical
experiments in one and two spatial dimensions demonstrate that ADR dynamics
admit intrinsically low-rank MPS representations with bounded bond-dimension
growth. The method shows excellent agreement with high-accuracy RK45
reference solutions and remains robust under synthetic stress tests that
isolate transport, diffusion, reaction, and oscillatory initial conditions.
The approach extends naturally to two spatial dimensions while preserving
stability and accuracy and maintaining compact tensor-network
representations.

\section{Background}
\label{sec:background}

This section introduces the tensor-network notation used throughout the paper, with a focus on MPS and MPOs. Originally developed in quantum many-body physics, these representations are well suited to structured discretizations of partial differential equations. We adopt a quantum-inspired Dirac notation to express tensorized PDE solution fields in a form analogous to states in a composite tensor-product space.

\paragraph*{Tensorization and Quantum-Inspired Representation}
Consider a discretized PDE solution $u \in \mathbb{R}^{N}$ sampled on a uniform spatial grid. When $N = d^L$, we reshape $u$ into an order-$L$ tensor and interpret its entries as amplitudes of a quantum-inspired state
\begin{equation}
|\Psi\rangle
= \sum_{s_1,\dots,s_L}
\psi_{s_1\ldots s_L}
\, |s_1\rangle \otimes \cdots \otimes |s_L\rangle,
\qquad
\psi_{s_1\ldots s_L} \in \mathbb{R}.
\end{equation}
Here, each local basis vector $|s_i\rangle$ corresponds to a digit in a multi-index encoding of the spatial grid. In contrast to a true quantum system, the indices $s_i$ represent numerical degrees of freedom induced by the tensorization map rather than physical particles. The state $|\Psi\rangle$ resides in the tensor-product space
$\mathcal{H} = \bigotimes_{i=1}^{L} \mathbb{R}^{d}$, which is formally analogous to the state-space structure of quantum lattice systems. Although the amplitudes $\psi_{s_1\ldots s_L}$ contain exactly the same information as the original vector $u$, this representation exposes separability and correlation structures that can be efficiently exploited by tensor-network factorizations.

\paragraph*{Matrix Product States}
Matrix product states provide compact, structured representations of high-dimensional tensors of the form above \cite{Schollwock2011,Oseledets2011}. An MPS expresses the amplitudes as
$
\psi_{s_1\ldots s_L}
=
A_1^{(s_1)} A_2^{(s_2)} \cdots A_L^{(s_L)},
$
where each $A_i^{(s_i)} \in \mathbb{R}^{D_{i-1}\times D_i}$ is a site-local core, and $(D_0, D_1, \dots, D_L)$ are the bond dimensions with $D_0 = D_L = 1$. The bond dimensions $D_i$ quantify the amount of correlation (entanglement, in quantum terminology) across bond $i$. When these remain small, the representation requires only $O(L d \chi^2)$ parameters, where $\chi = \max_i D_i$, instead of the full $d^L$ degrees of freedom.

\paragraph*{Canonical Forms}
As in quantum many-body methods, numerical stability is improved by imposing orthonormality conditions on the virtual indices. In a left-canonical representation, the cores satisfy
$
\sum_{s_i} \left(A_i^{(s_i)}\right)^T A_i^{(s_i)} = I_{D_i},
$
ensuring that the left block of the network forms an orthonormal basis. Right-canonical and mixed-canonical forms are defined analogously. In the mixed-canonical form, a center tensor is introduced whose singular values correspond to the Schmidt coefficients across a chosen bipartition. These coefficients play a key role in stable truncation during time evolution.

\paragraph*{Matrix Product Operators}
Many operators arising from discretized PDEs exhibit locality and tensor-product structure similar to those of lattice Hamiltonians. An MPO represents a linear operator $\mathcal{O} : \mathbb{R}^{d^L} \to \mathbb{R}^{d^L}$ as
\begin{equation}
\mathcal{O}
=
\sum_{\{s_i\},\{s'_i\}}
O_{s_1 \ldots s_L}^{\;\; s'_1 \ldots s'_L}
\, |s_1\rangle\!\langle s'_1|
\otimes \cdots \otimes
|s_L\rangle\!\langle s'_L|,
\end{equation}
with the coefficient tensor factorized across sites as
\begin{equation*}
O_{s_1 \ldots s_L}^{\;\; s'_1 \ldots s'_L}
=
\sum_{\alpha_1,\dots,\alpha_{L-1}}
W_1^{(s_1,s'_1)}{}_{1,\alpha_1}
W_2^{(s_2,s'_2)}{}_{\alpha_1,\alpha_2}
\cdots
W_L^{(s_L,s'_L)}{}_{\alpha_{L-1},1},
\end{equation*}
where the $W_i$ are MPO cores and the indices $\alpha_i$ carry the operator
bond dimension. In the ADR setting, standard finite-difference advection and
diffusion stencils couple only neighboring grid points, leading to MPO bond
dimensions that remain small and independent of grid size.

\paragraph*{Applying an MPO to an MPS}
Given an MPS with cores $A_i^{(s_i)}$, the application
$|\widetilde{\Psi}\rangle = \mathcal{O}|\Psi\rangle$ yields a new MPS with cores
$
\widetilde{A}_i^{(s_i)}
=
\sum_{s'_i}
W_i^{(s_i,s'_i)} \otimes A_i^{(s'_i)},
$
where $\otimes$ denotes the Kronecker product combining the virtual bond spaces of the MPO and MPS cores. As a result, the bond dimensions increase multiplicatively with the MPO bond dimensions and are subsequently reduced by canonicalization and SVD-based truncation. This mechanism underlies each step of the tensor-network time integrator.

\begin{remark}[Relevance to ADR Time Stepping]
In the proposed framework, advection and diffusion operators are represented as MPOs acting on an MPS encoding of the solution field, while the reaction term is a diagonal on-site operator. Each explicit Euler step therefore reduces to a sequence of MPO--MPS contractions followed by canonicalization and controlled bond-dimension truncation.
\end{remark}

\section{Tensor-Network Framework for ADR Equations}
\label{sec:QTN}

We now specialize the tensor-network formalism to the numerical solution of ADR equations. Let $u(x,t)$ denote the solution on a spatial domain $\Omega \subset \mathbb{R}^{p}$, where the advection term models transport, the diffusion term accounts for dissipative spreading, and the reaction term describes local growth or decay. In the form considered here, the ADR equation is given by
\begin{equation}
u_t + c \cdot \nabla u = \nu \Delta u + \lambda u,
\end{equation}
where $c = (c_1,\dots,c_p)$ is a constant advection velocity vector, $\nu > 0$ is the diffusion coefficient, and $\lambda \in \mathbb{R}$ is the reaction rate, with $\lambda>0$ corresponding to growth and $\lambda<0$ to decay.
After spatial discretization, the central idea is to represent the solution field in MPS form and the differential operators in MPO form, so that time stepping can be carried out entirely within the tensor-network representation.

Let the discretized solution at time $t_k$ be written in the quantum-inspired form introduced in Section~\ref{sec:background},
$
|\Psi_k\rangle
=
\sum_{s_1,\ldots,s_L}
\psi^{(k)}_{s_1\ldots s_L}
\, |s_1\rangle \otimes \cdots \otimes |s_L\rangle,
$
where $(s_1,\ldots,s_L)$ is a base-$d$ encoding of the spatial grid index.
Following Section~\ref{sec:background}, we represent $|\Psi_k\rangle$ in MPS form, with controlled truncation applied during time stepping:
$
\psi^{(k)}_{s_1\ldots s_L}
\approx
A_1^{(s_1)} A_2^{(s_2)} \cdots A_L^{(s_L)},
$
where $A_i^{(s_i)} \in \mathbb{R}^{D_{i-1} \times D_i}$ and $D_0 = D_L = 1$. Canonical forms are used throughout to ensure numerical stability and enable efficient bond-dimension truncation during time stepping.

This construction extends naturally to multidimensional grids. For a $2^{L_x} \times 2^{L_y}$ grid, we map the two-dimensional index $(\ell_x,\ell_y)$ to a one-dimensional index using a row-major ordering,
$
n = \ell_x + 2^{L_x}\,\ell_y,
$
with $\ell_x \in \{0,\dots,2^{L_x}-1\}$ and $\ell_y \in \{0,\dots,2^{L_y}-1\}$. The resulting vector is then tensorized and represented as an MPS of length $L = L_x + L_y$. This ordering is used consistently in both the state encoding and the construction of the finite-difference MPOs. Such a representation retains the locality structure of the underlying finite-difference operators, which enables their encoding as MPOs with small, grid-independent bond dimensions \cite{KazeevKhoromskij2012}.

The ADR operator consists of three components: advection, diffusion, and reaction. At the discrete level, first-derivative operators such as $u_x$ and $u_y$, as well as second-derivative operators such as $u_{xx}$ and $u_{yy}$, have local finite-difference stencil structure and can therefore be encoded as MPOs. We denote these by $\mathcal{D}_x$, $\mathcal{D}_y$, $\mathcal{D}_{xx}$, and $\mathcal{D}_{yy}$ in two spatial dimensions, with analogous notation in higher dimensions. In $p$ spatial dimensions, we write the corresponding first- and second-derivative MPOs generically as $\mathcal{D}_j$ and $\mathcal{D}_{jj}$, respectively. For instance,
$
\mathcal{D}_x |\Psi_k\rangle
=
\sum_{s_1,\ldots,s_L}
\left(
\sum_{s'_1,\ldots,s'_L}
D_{x,\,s_1\ldots s_L}^{\;\;s'_1\ldots s'_L}
\psi^{(k)}_{s'_1\ldots s'_L}
\right)
|s_1 \ldots s_L\rangle,
$
where the coefficient tensor $D_x$ admits an MPO factorization across sites. The reaction term acts pointwise on the grid and is therefore represented by the diagonal operator
$
\mathcal{R}=\lambda I.
$

Let $\Delta t$ denote the time step. In tensor-network form, a single explicit Euler update is given by
$
|\Psi_{k+1}\rangle
=
|\Psi_k\rangle
+
\Delta t
\left[
- \sum_{j=1}^{p} c_j \mathcal{D}_j |\Psi_k\rangle
+
\nu \sum_{j=1}^{p} \mathcal{D}_{jj} |\Psi_k\rangle
+
\lambda |\Psi_k\rangle
\right].
$
Thus, each time step consists of a sequence of MPO--MPS contractions corresponding to transport, diffusion, and reaction. These operations generally increase the intermediate bond dimensions, so after each update the state is brought into mixed-canonical form and compressed via singular-value truncation.

We denote the projection onto the admissible MPS manifold with tolerance $\varepsilon$ and maximum bond dimension $\chi_{\max}$ by $\Pi_{\chi_{\max},\varepsilon}(\cdot)$. The resulting practical update rule is
$
|\Psi_{k+1}\rangle
=
\Pi_{\chi_{\max},\varepsilon}
\left(
|\Psi_k\rangle
+
\Delta t\, \mathcal{L}_{\mathrm{ADR}} |\Psi_k\rangle
\right),
$
where $\mathcal{L}_{\mathrm{ADR}}$ denotes the combined MPO associated with the full ADR operator. This update is conceptually similar to tensor-network time-evolution methods for lattice systems, in which operator applications are interleaved with controlled truncation \cite{Vidal2004}. For visualization and error evaluation, the physical grid solution is reconstructed from the tensor-network state via
$
u(x_i,t_{k+1})
=
\big(\mathrm{decode}(|\Psi_{k+1}\rangle)\big)[i],
$
while all other operations are performed directly in the tensor-network representation.

\paragraph*{Example: finite-difference MPO}
Consider the one-dimensional second-order centered finite-difference approximation of the Laplacian,
$
(\mathcal{D}_{xx} u)_i
=
\frac{1}{\Delta x^2}\left(u_{i-1} - 2u_i + u_{i+1}\right),
$
which corresponds to a tridiagonal operator with stencil $[1,-2,1]$. Due to its nearest-neighbor coupling structure, this operator admits a compact MPO representation with small constant bond dimension (e.g., 3). The MPO cores propagate identity and shift contributions along the tensor network, while boundary conditions are incorporated through appropriate modifications of the edge cores. Explicit constructions of such finite-difference MPOs are available in \cite{KazeevKhoromskij2012}.

\section{Methodology}
\label{sec:method}

This section describes the computational procedure used to advance the ADR equation directly in matrix product state/operator (MPS/MPO) form. All computations are carried out in the tensorized representation introduced in Sections~\ref{sec:background}--\ref{sec:QTN}.

At each time step, the ADR update is computed through a sequence of MPO--MPS contractions followed by canonicalization and controlled truncation. Starting from the encoded state $|\Psi_k\rangle$, we form the intermediate update
\[
|\widetilde{\Psi}_{k+1}\rangle
=
|\Psi_k\rangle
+
\Delta t\,\mathcal{L}_{\mathrm{ADR}}\,|\Psi_k\rangle,
\]
where $\mathcal{L}_{\mathrm{ADR}}$ denotes the combined MPO associated with the discretized advection, diffusion, and reaction terms. This operation constitutes the main algebraic step of the method.

The choice of explicit Euler is motivated by its simplicity and direct compatibility with tensor-network representations, where each time step reduces to a sequence of MPO--MPS contractions. Although this requires a relatively large number of time steps, the per-step computational cost remains low due to the small bond dimensions observed in practice. In contrast, higher-order methods such as RK45 involve multiple operator evaluations per step and are less straightforward to implement efficiently in tensor-network form; the development of higher-order tensor-network integrators is therefore left for future work.

Since applying an MPO to an MPS enlarges the virtual bond dimensions, the intermediate state is subsequently brought into mixed-canonical form and compressed using singular-value decompositions (SVDs).
More precisely, at each bond we discard singular values smaller than a prescribed tolerance $\varepsilon_{\mathrm{SVD}}$ and cap the retained bond dimension by $\chi_{\max}$. We denote the resulting projection onto the admissible MPS manifold by
$
\Pi_{\chi_{\max},\varepsilon_{\mathrm{SVD}}}(\cdot).
$
The practical time-stepping rule is therefore
\[
|\Psi_{k+1}\rangle
=
\Pi_{\chi_{\max},\varepsilon_{\mathrm{SVD}}}
\!\left(
|\Psi_k\rangle
+
\Delta t\,\mathcal{L}_{\mathrm{ADR}}\,|\Psi_k\rangle
\right).
\]
This projection step controls representation complexity and mitigates bond-dimension growth during repeated operator application.
Starting from an initial encoded state $|\Psi_0\rangle$, the above update is applied sequentially to generate the full rollout in tensor-network form. Physical-space fields are reconstructed only when needed for visualization or comparison with reference solvers
$
u(x_i,t_k)=\mathrm{decode}\!\left(|\Psi_k\rangle\right)[i].
$
All other operations, including state evolution and derivative application, are performed entirely in MPS/MPO form. In the numerical experiments, the MPOs are constructed directly from the finite-difference discretization of the ADR operator, and the same update structure is used at every time step.
The following proposition provides a qualitative bound on the accumulation of truncation errors under repeated projection in the tensor-network time-stepping scheme.

\begin{proposition}[Propagation of truncation errors]
Let
$
F(\Psi):=\Psi+\Delta t\,\mathcal{L}_{\mathrm{ADR}}\Psi
$
denote the one-step explicit Euler map, and consider the projected scheme
$
\Psi_{k+1}
=
\Pi_{\chi_{\max}, \varepsilon_{\mathrm{SVD}}}
\big(F(\Psi_k)\big).
$
Let the corresponding untruncated trajectory be defined by
$
\widehat{\Psi}_{k+1}=F(\widehat{\Psi}_k)$,
$\widehat{\Psi}_0=\Psi_0$.
Assume that for all intermediate states $v$,
$
\|\Pi_{\chi_{\max},\varepsilon_{\mathrm{SVD}}}(v)-v\|
\le \varepsilon_{\mathrm{SVD}},
$
and that $F$ is Lipschitz continuous with constant $q \ge 0$, i.e.,
$
\|F(u)-F(v)\|\le q\,\|u-v\| \quad \forall\,u,v$.
Then the global truncation error $e_k := \Psi_k - \widehat{\Psi}_k$ satisfies
$
\|e_{k+1}\|\le q\,\|e_k\|+\varepsilon_{\mathrm{SVD}}.
$
In particular, if $e_0 = 0$, then
$
\|e_k\|
\le
\begin{cases}
k\,\varepsilon_{\mathrm{SVD}}, & q = 1, \\[1ex]
\dfrac{1 - q^k}{1 - q}\,\varepsilon_{\mathrm{SVD}}, & 0 \le q < 1
\end{cases}.
$
Moreover, if the ADR dynamics are contractive and the explicit Euler map satisfies
$
\|F(u) - F(v)\| \le (1 - \gamma \Delta t)\, \|u - v\|
\quad \text{for some } \gamma > 0,
$
then
$
\|e_k\|
\le
\dfrac{1 - (1 - \gamma \Delta t)^k}{\gamma \Delta t}\,
\varepsilon_{\mathrm{SVD}}
\le
\dfrac{\varepsilon_{\mathrm{SVD}}}{\gamma \Delta t}
$.
Thus, in dissipative ADR regimes, truncation errors remain uniformly bounded in time.
\end{proposition}

\begin{proof}
From the definitions of the projected and untruncated schemes, we have
$\Psi_{k+1}
=
\Pi_{\chi_{\max}, \varepsilon_{\mathrm{SVD}}}\big(F(\Psi_k)\big),$ and
$\widehat{\Psi}_{k+1}
=
F(\widehat{\Psi}_k)$.
Hence,
$e_{k+1}
=
\Pi_{\chi_{\max}, \varepsilon_{\mathrm{SVD}}}\big(F(\Psi_k)\big)
-
F(\widehat{\Psi}_k)$.
Adding and subtracting $F(\Psi_k)$ yields
$e_{k+1}
=
\Big(
\Pi_{\chi_{\max}, \varepsilon_{\mathrm{SVD}}}\big(F(\Psi_k)\big)
-
F(\Psi_k)
\Big)
+
\Big(
F(\Psi_k) - F(\widehat{\Psi}_k)
\Big)$.
Taking norms and using the triangle inequality gives
$\|e_{k+1}\|
\le
\left\|
\Pi_{\chi_{\max}, \varepsilon_{\mathrm{SVD}}}\big(F(\Psi_k)\big)
-
F(\Psi_k)
\right\|
+
\|F(\Psi_k) - F(\widehat{\Psi}_k)\|.$
By the truncation assumption,
$\left\|
\Pi_{\chi_{\max},\varepsilon_{\mathrm{SVD}}}\!\big(F(\Psi_k)\big)
-
F(\Psi_k)
\right\|
\le
\varepsilon_{\mathrm{SVD}}$,
and by Lipschitz continuity,
$\|F(\Psi_k)-F(\widehat{\Psi}_k)\|
\le
q\,\|e_k\|$.
This yields
$\|e_{k+1}\|\le q\,\|e_k\|+\varepsilon_{\mathrm{SVD}}$.
Assuming $e_0=0$, iteration gives
$\|e_k\|
\le
\sum_{j=0}^{k-1} q^j \, \varepsilon_{\mathrm{SVD}}$.
If $q = 1$, this reduces to $\|e_k\| \le k\, \varepsilon_{\mathrm{SVD}}$. If $0 \le q < 1$, summing the geometric series gives
$\|e_k\|
\le
\dfrac{1 - q^k}{1 - q}\, \varepsilon_{\mathrm{SVD}}$.
If $q = 1 - \gamma \Delta t$ with $\gamma > 0$, then
$\|e_k\|
\le
\sum_{j=0}^{k-1} (1 - \gamma \Delta t)^j \, \varepsilon_{\mathrm{SVD}}
=
\dfrac{1 - (1 - \gamma \Delta t)^k}{\gamma \Delta t}\,
\varepsilon_{\mathrm{SVD}}
\le
\dfrac{\varepsilon_{\mathrm{SVD}}}{\gamma \Delta t}$.
\end{proof}

\begin{remark}[Implementation Details]
The tensor-network time-stepping procedure follows a fixed sequence of MPO--MPS contraction, canonicalization, and rank truncation. No learning or data-driven optimization is involved: the MPOs are determined entirely by the chosen spatial discretization, while truncation is used solely to control representation complexity.
\end{remark}

\begin{remark}[Stability and Complexity]
The tensor-network update inherits the stability properties of the underlying time-integration scheme, here explicit Euler with a CFL-restricted time step, together with the dissipative effects induced by diffusion and reaction. The projection $\Pi_{\chi_{\max},\varepsilon_{\mathrm{SVD}}}$ acts as a complexity-control mechanism and helps maintain bounded MPS ranks when the evolving solution remains approximately low-rank.
\end{remark}

\begin{remark}[Computational Scaling]
The computational cost per time step scales linearly with the number of MPS sites and polynomially with the MPS and MPO bond dimensions. Since the number of tensor sites grows with the tensorized representation rather than directly with the full spatial dimension, the method can remain efficient when the bond dimensions stay small. In contrast, classical grid-based solvers typically scale directly with the number of spatial degrees of freedom. The present study therefore emphasizes representation efficiency in regimes where the solution admits a compact low-rank tensor-network structure, rather than worst-case complexity guarantees.
\end{remark}

\section{Numerical Experiments} \label{sec:experiments}
In this section, we evaluate the proposed tensor-network framework on
ADR problems in one and two spatial
dimensions, comparing against high-accuracy classical solvers.
\subsection{One-Dimensional Advection--Diffusion--Reaction}

We consider the one-dimensional ADR equation
\begin{equation}
\label{eq:adr}
u_t + c\,u_x = \nu\,u_{xx} + \lambda\,u,
\qquad x\in[0,1],\; t\in[0,T],
\end{equation}
in the dissipative regime \(\lambda<0\). The spatial domain is discretized
with \(N_x=2^L\) grid points, where \(L=9\), giving \(N_x=512\) and
\(\Delta x\approx 1.95\times 10^{-3}\). Homogeneous Dirichlet boundary conditions $u(0,t)=u(1,t)=0$ are imposed and enforced via a mask MPS projection applied after each time step, which zeroes the boundary entries while preserving the tensor-network structure. The initial condition is taken as the localized Gaussian
$
u(x,0)=\exp\!\big(-100(x-0.3)^2\big).
$
Unless otherwise stated, we use $c=0.5$, $\nu=0.01$, $\lambda=-2.0$, and $T=1.0$, corresponding to rightward transport, diffusive spreading, and exponential decay.
The discretized state \(u(\cdot,t)\in\mathbb{R}^{N_x}\) is reshaped into a
tensor of order \(L\) and encoded as an MPS with physical dimension \(d=2\).
The derivative operators \(u_x\) and \(u_{xx}\) are represented as MPOs
constructed from second-order finite-difference stencils and scaled by
\(\Delta x^{-1}\) and \(\Delta x^{-2}\), respectively. Boundary conditions are
enforced at each step through a mask-MPS projection. Time integration is
performed entirely in MPS/MPO form using explicit Euler,
$
u^{k+1}=u^k+\Delta t\big(-c\,D_xu^k+\nu\,D_{xx}u^k+\lambda\,u^k\big),
$
followed by canonicalization and SVD truncation with tolerance \(10^{-12}\)
and maximum bond dimension \(60\). Dense reconstruction is used only for
visualization and comparison.

The time step is chosen conservatively as
$
\Delta t \le \min\!\left(\frac{\Delta x}{|c|},\frac{\Delta x^2}{\nu},
\frac{1}{|\lambda|}\right),
$
which gives \(\Delta t\approx 10^{-4}\). As a reference solution, we use a
classical method-of-lines discretization with adaptive \texttt{RK45} on the
same spatial grid. For comparison, the QTN trajectory is interpolated onto
the RK45 time points.

Figure~\ref{fig:adr_qtn} compares the QTN and RK45 solutions. The QTN solver reproduces the expected advection, diffusive broadening, and exponential decay with high fidelity. The difference $u_{\mathrm{RK45}} - u_{\mathrm{QTN}}$ remains smooth and localized around the transported pulse, while the maximum error stays small and well controlled over time. Global observables, including the peak amplitude and total mass, are also accurately captured, indicating that the tensor-network representation preserves both local solution structure and overall decay dynamics.

\begin{figure}[t]
    \centering
    \includegraphics[width=0.5\textwidth]{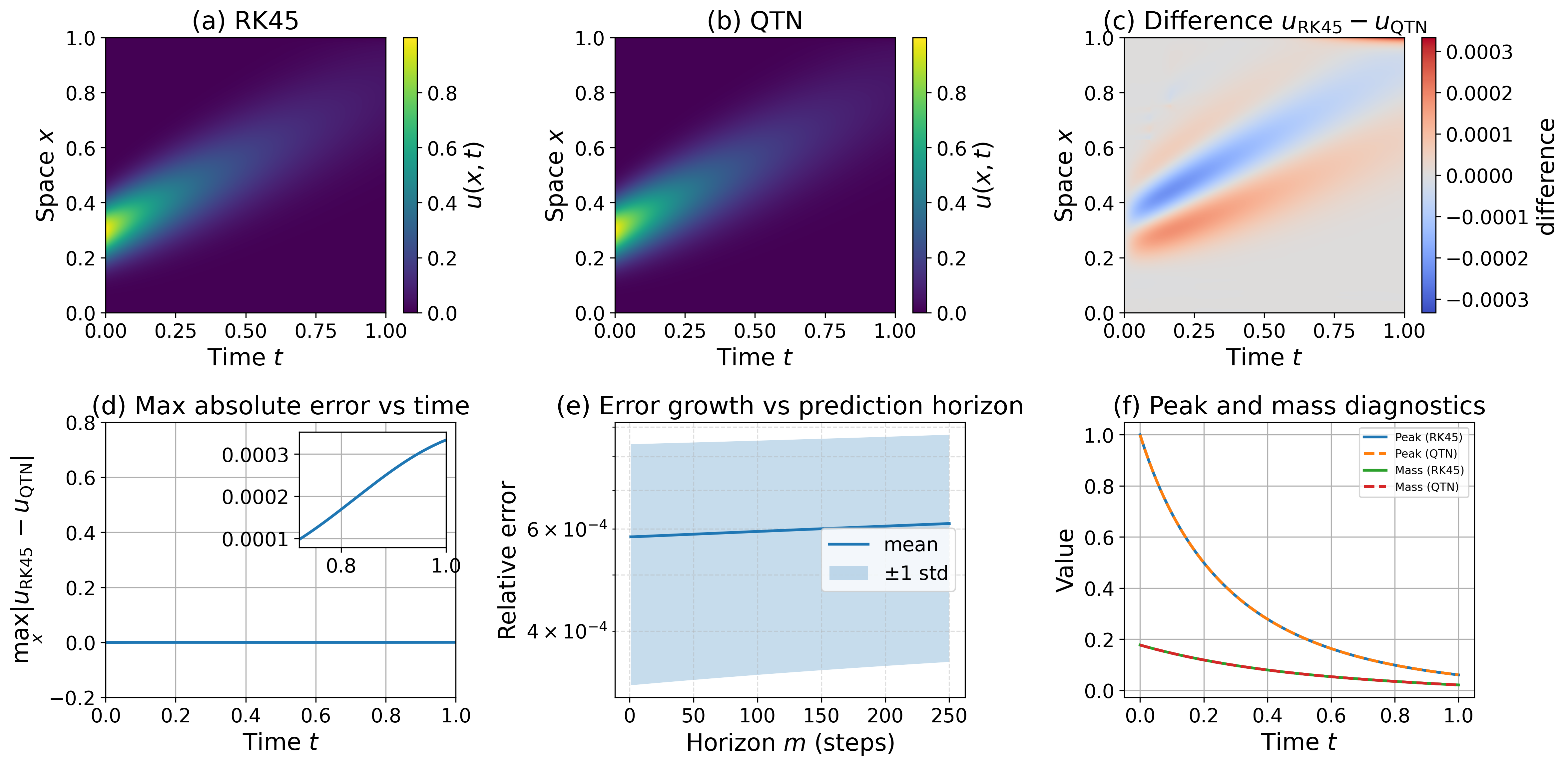}
    \caption{\small
Comparison of QTN and RK45 solutions for the 1D ADR equation. 
(a,b) RK45 reference and QTN solution, with the QTN trajectory interpolated onto the RK45 time grid. 
(c) Signed difference $u_{\text{RK45}} - u_{\text{QTN}}$, which remains smooth and localized. 
(d) Maximum absolute error over space versus time (inset: late-time behavior). 
(e) Restart-averaged error growth versus prediction horizon. 
(f) Peak amplitude and total mass, showing that the QTN method accurately preserves both local and global decay dynamics.
    }
    \label{fig:adr_qtn}
\end{figure}

To assess representation complexity, we track the mean bond dimension $\bar{\chi}(t)$ and the spatial rank profile $\chi_i(t)$ throughout the simulation. As shown in Figure~2, the mean bond dimension remains very small, typically $\bar{\chi}(t) \approx 3$–$4$, and the full rank profile stays uniformly bounded across the tensor cores. This confirms that the one-dimensional ADR dynamics admit a compact low-rank QTT/MPS representation and that truncation remains stable throughout the rollout.

In particular, the boundedness of the ranks reflects the fact that the underlying solution remains smooth and does not generate long-range correlations. While transport mechanisms such as advection can increase the bond dimension by translating localized features across the grid and inducing additional inter-core correlations in the tensorized representation, this growth remains limited because the transported structures remain smooth and localized.

\begin{figure}[t]
    \centering
    \includegraphics[width=0.45\textwidth]{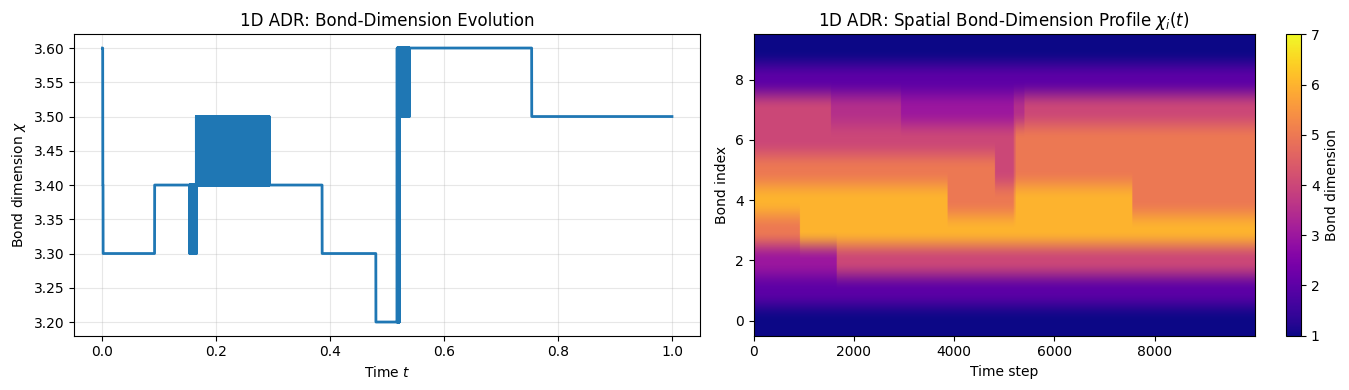}
    \caption{\small
Bond-dimension diagnostics for the 1D ADR simulation.
Left: mean bond dimension $\bar{\chi}(t)$ over time, showing that the tensor-network representation remains compact and stable.
Right: spatial bond-dimension profile $\chi_i(t)$ (heatmap), illustrating variation of MPS ranks across tensor cores and time.
The ranks remain uniformly small (typically $\chi \approx 3$--$4$), confirming the intrinsically low-rank structure of the dynamics.
    }
    \label{fig:adr1d_bond}
\end{figure}

We further examine six synthetic stress tests isolating transport, diffusion, reaction, and nonsmooth initial data. Figure~\ref{fig:adr1d_stress} shows the resulting trajectories of the mean bond dimension $\bar{\chi}(t)$. Pure advection produces the largest rank growth, with $\bar{\chi}(t)$ remaining below approximately 6. This behavior arises because advection translates the solution profile across the grid, which in the tensorized representation induces additional coupling between tensor cores and increases inter-core correlations.

Diffusion and reaction decay yield nearly constant or decreasing rank, reflecting their dissipative nature, which smooths the solution and reduces correlation complexity. Reaction growth also remains low-rank because it amplifies the solution pointwise without introducing additional spatial correlations. The advection–diffusion case exhibits only mild fluctuations due to the competing effects of transport and smoothing, while high-frequency initial data produce a transient increase in rank before being rapidly regularized.
Across all test cases, the ranks remain small and bounded, demonstrating that the tensor-network representation is robust across transport-dominated, diffusion-dominated, reaction-driven, and oscillatory regimes.

\begin{figure}[t]
    \centering
    \includegraphics[width=0.45\textwidth]{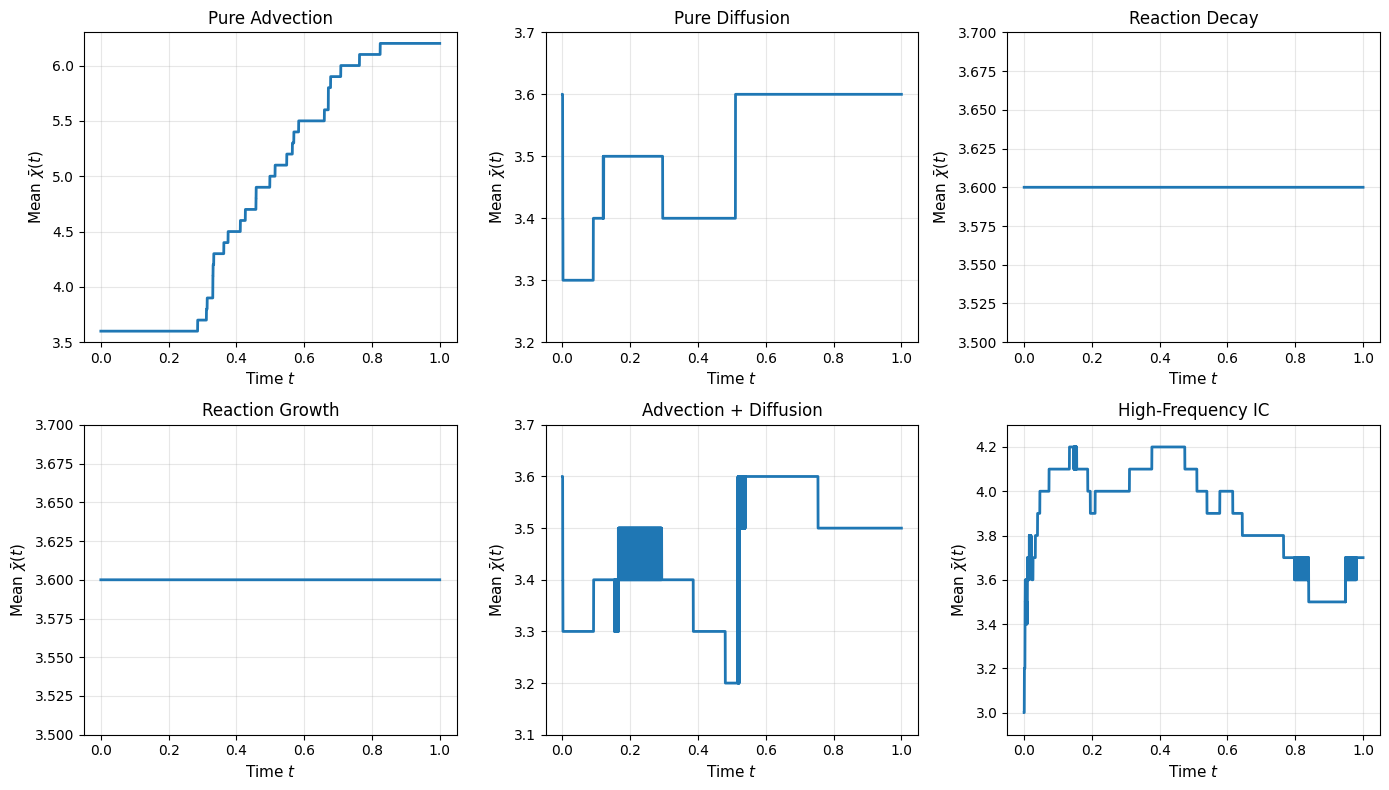}
    \caption{ \small
 Mean bond dimension $\bar{\chi}(t)$ for six 1D synthetic stress tests: 
(top) pure advection, pure diffusion, reaction decay; 
(bottom) reaction growth, advection--diffusion, and a high-frequency initial condition. 
Across all regimes, the representation remains low-rank and stable. 
Advection produces the largest rank growth due to translation, while diffusion and reaction decay yield nearly constant or decreasing ranks. 
Reaction growth remains low-rank, and oscillatory data are rapidly regularized.
    }
    \label{fig:adr1d_stress}
\end{figure}

\subsection{Two-Dimensional Advection--Diffusion--Reaction}
We next consider the two-dimensional ADR equation on the unit square
$(x,y)\in[0,1]^2$,
\begin{equation}
\partial_t u + c_x\,\partial_x u + c_y\,\partial_y u
= \nu(\partial_{xx}u+\partial_{yy}u) + \lambda u,
\end{equation}
with homogeneous Dirichlet boundary conditions $u=0$ on $\partial\Omega$.
Boundary conditions are enforced using the same mask-MPS projection as in the one-dimensional case.
We use $\nu=0.01$, $c_x=0.5$, $c_y=0.2$, $\lambda=-2.0$, and final time
$T=1$, with initial condition
$
u_0(x,y)=\exp\!\big(-80[(x-0.3)^2+(y-0.4)^2]\big).
$
The spatial domain is discretized on a tensor grid of size $N_x \times N_y = 2^{L_x} \times 2^{L_y}$ with $L_x = L_y = 6$. The full $N_x N_y$ degrees of freedom are represented as an MPS of length $L_x + L_y$ using the row-major ordering introduced in Section~\ref{sec:QTN}, while the derivative operators are encoded as MPOs.
Time integration is performed using explicit Euler with a CFL-type step-size restriction, and the reference solution is computed using a dense method-of-lines discretization with adaptive RK45.

Figure~\ref{fig:adr2d_snapshots} compares QTN and RK45 snapshots at representative times. The QTN solution accurately reproduces diagonal advection, diffusive spreading, and reaction-driven decay with high fidelity. The difference fields remain small and structured, with magnitude of order $10^{-3}$, indicating a mild phase shift rather than significant amplitude distortion.

To quantify the discrepancy, we track the maximum pointwise error
$
E_\infty(t)=\max_{x,y}\big|u_{\mathrm{RK45}}(x,y,t)-u_{\mathrm{QTN}}(x,y,t)\big|
$
and the cumulative maximum error
$
E_{\mathrm{cum}}(t_k)=\max_{j\le k}E_\infty(t_j).
$
As shown in Figure~\ref{fig:adr2d_diagnostics}, the pointwise error remains below $7.17 \times 10^{-4}$ at the final time and below $3.21 \times 10^{-3}$ throughout the rollout, while the cumulative error saturates early and remains bounded. Global observables, including the peak amplitude and total mass, are also captured to high accuracy.
These results demonstrate that the QTN framework extends naturally to two spatial dimensions while preserving both accuracy and stability.

\begin{figure}[t]
    \centering
    \includegraphics[width=\linewidth]{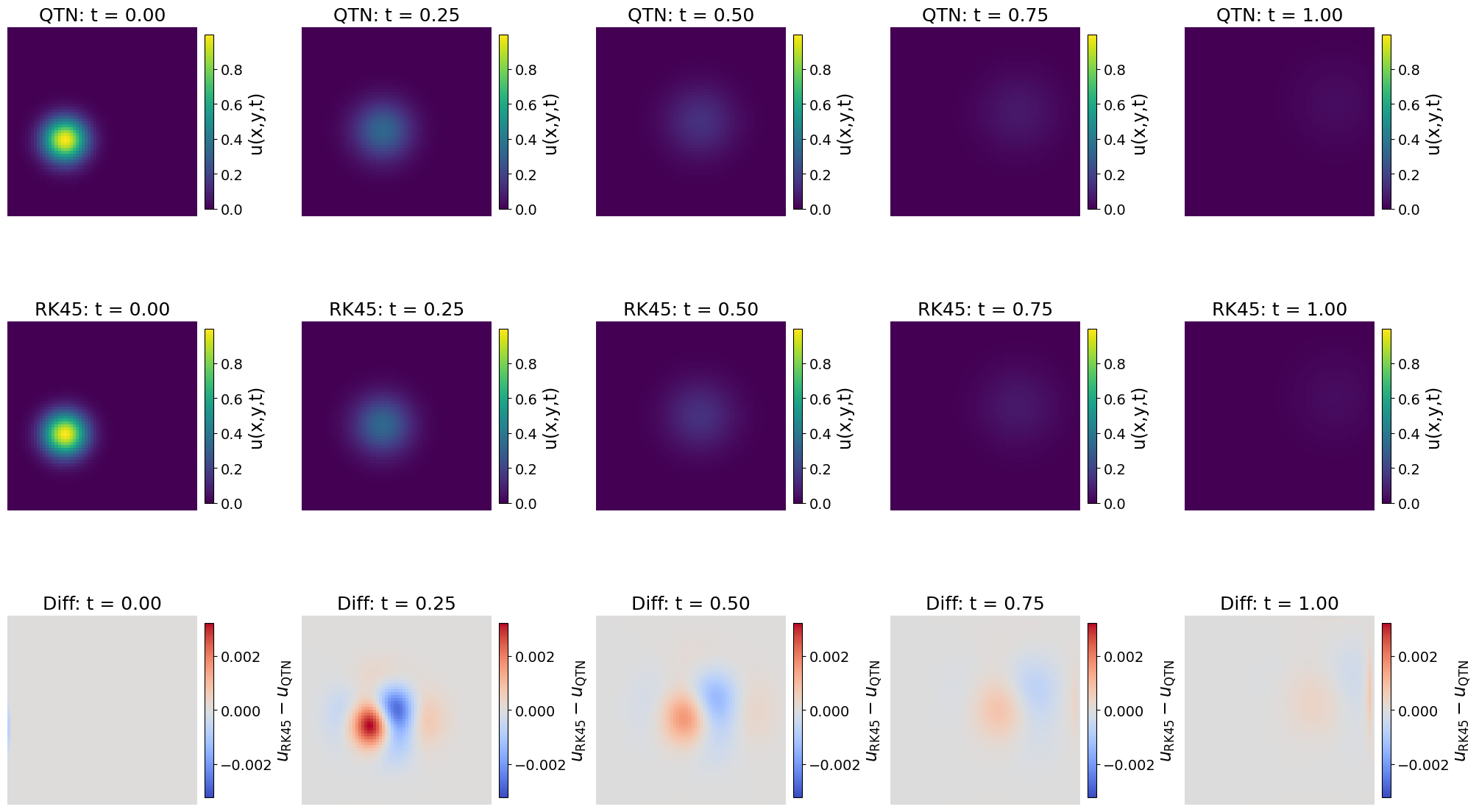}
    \caption{\small
    Snapshots of the 2D ADR solution at selected times. 
Top: QTN solution. Middle: RK45 reference. Bottom: difference $u_{\mathrm{RK45}} - u_{\mathrm{QTN}}$. 
The Gaussian pulse advects diagonally while diffusing and decaying, and the QTN solution closely matches the reference. 
The difference remains small and localized around the transported pulse.    
    }
    \label{fig:adr2d_snapshots}
\end{figure}

\begin{figure}[t]
    \centering
    \includegraphics[width=\linewidth]{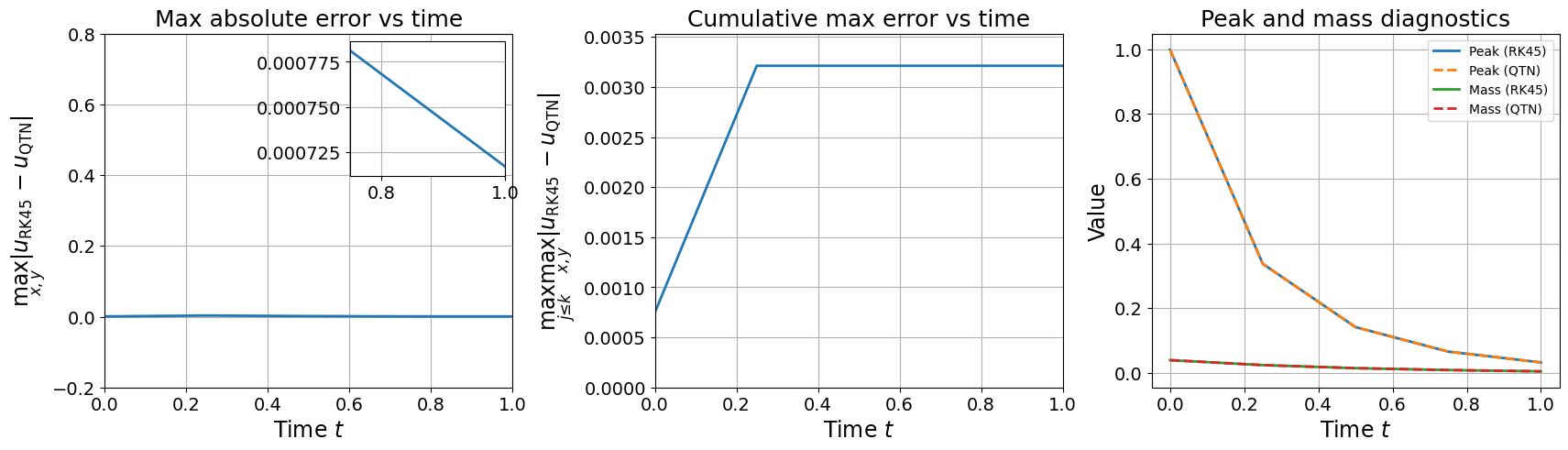}
    \caption{\small
Diagnostics for the 2D ADR problem. 
Left: maximum pointwise error $E_\infty(t)$ (inset: late-time behavior). 
Middle: cumulative maximum error $E_{\mathrm{cum}}(t)$. 
Right: peak amplitude and total mass for QTN and RK45. 
Errors remain small and bounded, and the QTN solution accurately reproduces the global decay of the reference.
    }
    \label{fig:adr2d_diagnostics}
\end{figure}

To assess representation complexity, we track the mean bond dimension $\bar{\chi}(t)$ and the spatial rank profile $\chi_i(t)$ throughout the simulation. As shown in Figure~\ref{fig:adr2d_bond}, the mean bond dimension remains very small, typically $\bar{\chi}(t) \approx 3$--$3.6$, and the ranks remain uniformly bounded across the tensor cores. This indicates that the two-dimensional ADR dynamics admit a compact low-rank QTN representation.
The bounded ranks reflect the smooth and localized structure of the solution, which does not generate significant long-range correlations in the tensorized representation.

\begin{figure}[t]
    \centering
    \includegraphics[width=0.45\textwidth]{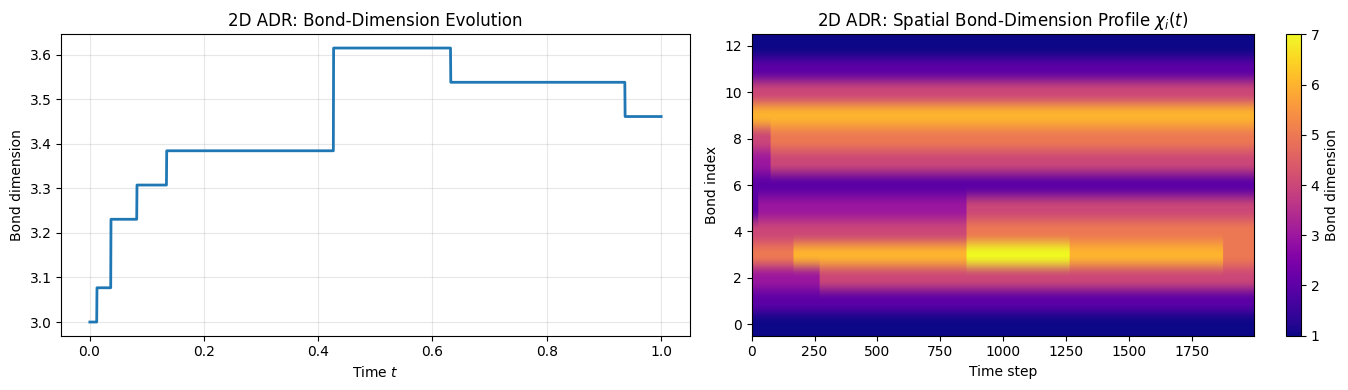}
    \caption{\small
Bond-dimension diagnostics for the 2D ADR simulation. 
Left: mean bond dimension $\bar{\chi}(t)$ over time, showing that the representation remains compact and stable. 
Right: spatial bond-dimension profile $\chi_i(t)$ (heatmap), illustrating variation of MPS ranks along the tensor chain and over time. 
Uniformly small ranks confirm the intrinsically low-rank structure of the dynamics.
    }
    \label{fig:adr2d_bond}
\end{figure}

We examine six synthetic stress tests isolating transport, diffusion, reaction, and high-frequency initial data (Fig.~\ref{fig:adr2d_stress}). In pure advection, diffusion, and reaction, the mean bond dimension remains nearly constant around $\bar{\chi}(t) \approx 2.8$--$2.9$, indicating low correlation complexity.
In contrast to 1D, advection in 2D produces only minor rank growth, as the transported Gaussian remains smooth and approximately separable. The high-frequency initial condition induces a transient rank increase before rapidly stabilizing under diffusion.
Across all tests, ranks remain small and bounded, demonstrating robustness across transport-, diffusion-, reaction-, and oscillatory regimes.

\begin{figure}[t]
    \centering
    \includegraphics[width=0.45\textwidth]{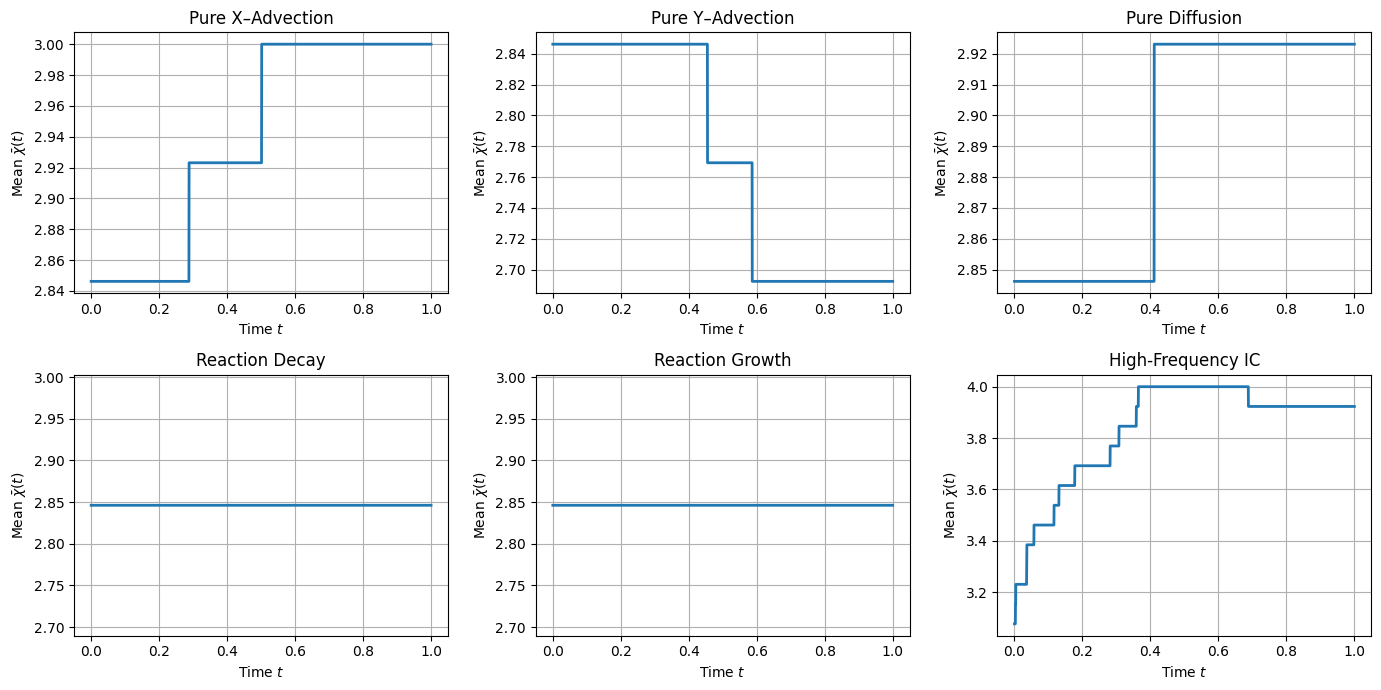}
    \caption{\small
 Mean bond dimension $\bar{\chi}(t)$ for six 2D synthetic stress tests: 
(top) pure $x$-advection, pure $y$-advection, pure diffusion; 
(bottom) reaction decay, reaction growth, and a high-frequency initial condition. 
In all cases, the rank remains nearly constant in the narrow range $\bar{\chi}(t)\approx 2.8$--$2.9$, reflecting the intrinsically low-rank structure. 
The high-frequency initial condition produces a transient rank increase that rapidly stabilizes under diffusion.
}
    \label{fig:adr2d_stress}
\end{figure}

\section{Conclusion}
\label{sec:conclusion}
We presented a quantum-inspired tensor-network framework for the numerical solution of linear ADR equations, with discretized solution fields encoded as MPS and finite-difference operators as MPOs. Time integration is performed entirely in tensor-network form using explicit Euler with controlled truncation.
Numerical experiments in one and two dimensions show that ADR dynamics admit compact low-rank tensor-network representations. Across transport-, diffusion-, reaction-, and oscillatory regimes, bond dimensions remain small and stable. The proposed solver agrees closely with RK45 reference solutions and accurately reproduces both local profiles and global observables.
These results demonstrate that tensor-network representations provide an effective structure-preserving approach for PDE simulation in regimes where the solution remains approximately low-rank, enabling accurate and stable time stepping in one and two dimensions.

\bibliographystyle{IEEEtran}%{IEEEconf}
\bibliography{IEEEabrv,Final}

\end{document}